\documentclass[portuges,12pt,letter]{article}
\usepackage[centertags]{amsmath}
\usepackage{amsfonts}
\usepackage{newlfont}
\usepackage{amscd}
\usepackage{graphics}
\usepackage{epsfig}
\usepackage{indentfirst}
\usepackage{amsxtra}
\usepackage[latin1]{inputenc}
\usepackage{amssymb, amsmath}
\usepackage{amsthm}
\usepackage[mathscr]{eucal}

\newtheorem{thm}{Theorem}[section]

\newtheorem{obe}[thm]{Remark}

\author{Fabio Silva Botelho \\ Department of Mathematics \\  Federal University of Santa Catarina, UFSC \\
Florian\'{o}polis, SC - Brazil}

\title{\bf  On duality principles for  one and three-dimensional non-linear models in elasticity}

\begin{document}
\date{}
\maketitle

\abstract{ In this article, we develop  duality principles applicable to primal variational formulations found in the non-linear elasticity theory.
As a first application, we establish the concerning results in details for  one and three-dimensional models. We emphasize such  duality
principles are  applicable to a larger class of variational optimization problems, such as non-linear models of plates and shells  and other
models in elasticity. Finally, we formally prove there is no duality gap between the primal and dual formulations, in a local extremal context.}

\section{Introduction}
In the first part of this  work, we develop a duality principle and sufficient conditions of local optimality for an one-dimensional model in non-linear elasticity.
The results are based on fundamental tools of convex analysis and duality theory.

About the references, this article in some sense extends and complements the original works of Telega, Bielski and their co-workers, \cite{85,2900,10,11}. In particular in \cite{85}, published in 1985
and in \cite{11}, for three-dimensional elasticity and related models, the authors established  duality principles and concerning global optimality conditions, for the special case in which the stress tensor is positive definite at a critical point. In this specific sense, the present work complements such  previous ones, considering we establish a sufficient condition for local minimality which does not require the stress tensor to be either positive or negative defined along the concerning domain. Such an optimality condition is summarized by the condition $\|u_x\|_\infty <1/4$ at a critical point.

The tools of convex analysis and duality  theory here used may be found in \cite{120,6,29}.
Existence of results in non-linear elasticity and related models may be found in \cite{903,3,4}.

Finally, details on the function spaces addressed may be found in \cite{1}.

At this point, we start to describe the primal variational formulation for the one-dimensional model.

Let $\Omega=[0,L] \subset \mathbb{R}$ be an interval which represents the axis of a straight bar of length $L$ and constant cross section area $A$.

We denote by $u:[0,L] \rightarrow \mathbb{R}$ the field of axial displacements for such a bar, resulting from the application of an axial load field
$P \in C([0,L]).$

We also denote $$U=\{u \in C^1([0,L])\;:\; u(0)=u(L)=0\},$$ and
$$\hat{U}=\{u \in U\;:\; \|u_x\|_\infty< 1/4\}.$$

The energy for such a system, denoted by $J:U \rightarrow \mathbb{R}$, is expressed as
$$J(u)=\frac{EA}{2}\int_0^L\left(u_x +\frac{1}{2}u_x^2\right)^2\;dx-\int_0^L Pu\;dx,\; \forall u \in U,$$
where $E>0$ is the Young modulus.

We shall also define
$$G(u_x)=\frac{EA}{2}\int_0^L \left(u_x+\frac{1}{2}u_x^2\right)^2\;dx.$$

Finally, generically we shall denote, for $u \in U$ and $r>0$, $$B_r(u)=\{v \in U\;:\; \|v-u\|_U<r\}$$
where $$\|v\|_U=\max_{x \in [0,L]} \{|v(x)|+|v_x(x)|\},\; \forall v \in U.$$

Moreover, defining $V=C([0,L])$, for $z^* \in V$ and $r_1>0$, we shall generically also denote 
$$B_{r_1}(z^*)=\{v \in V\;:\; \|v-z^*\|_V<r_1\},$$
where
$$\|v\|_V=\|v\|_\infty=\max_{x \in [0,L]} |v(x)|,\; \forall v \in V.$$

Similar corresponding standard notations are valid for $V\times V$ and the 3-dimensional model.
\section{The main duality principle for the one-dimensional model}
Our first duality principle is summarized by the following theorem.

\begin{thm}Let $J:U \rightarrow \mathbb{R}$ be defined by
$$J(u)=\frac{EA}{2}\int_0^L\left(u_x +\frac{1}{2}u_x^2\right)^2\;dx-\int_0^L Pu\;dx.$$

Assume $u_0 \in \hat{U}$ is such that $$\delta J(u_0)=\mathbf{0}.$$

Define $F:U \rightarrow \mathbb{R}$ by $$F(u_x)=\frac{K}{2}\int_0^L u_x^2\;dx,$$

$G_K:U \rightarrow \mathbb{R}$ by
$$G_K(u_x)=G(u_x)+\frac{K}{2}\int_0^L u_x^2\;dx$$
and
$J^*:V \times V \times V \rightarrow \mathbb{R}$ by
$$J^*(v^*,z^*)=F^*(z^*)-G^*_K(v^*,z^*),$$
\begin{eqnarray}F^*(z^*)&=& \sup_{v \in V} \{\langle v,z^*\rangle_{L^2}-F(v)\} \nonumber \\ &=&
\frac{1}{2K} \int_0^L (z^*)^2\;dx\end{eqnarray}
and
\begin{eqnarray} G^*_K(v^*,z^*)&=& \sup_{ (v_1,v_2) \in V \times V} \left\{
\langle v_1,z^*+v_2^*\rangle_{L^2}+\langle v_2,v_1^*\rangle_{L^2} \right.\nonumber \\
&&\left.-\frac{EA}{2}\int_0^L\left(v_1+\frac{1}{2}v_2^2\right)^2\;dx- \frac{K}{2}\int_0^L v_2^2\;dx\right\}
\nonumber \\ &=& \frac{1}{2}\int_0^L \frac{(v_1^*)^2}{z^*+v_2^*+K}\;dx+\frac{1}{2EA}\int_0^L (v_2^*+z^*)^2\;dx,\end{eqnarray}
if $v_2^*+z^*+K>0, \text{ in } \Omega.$

Define also,
$$A^*=\{v^*=(v_1^*,v_2^*) \in V \times V\;:\; (v_1^*)_x+(v_2^*)_x+P=0,\text{ in } \Omega\},$$

$$K=EA/2,$$
$$\hat{z}^*=K (u_0)_x,$$
$$\hat{v}_2^*=EA\left((u_0)_x+\frac{1}{2}(u_0)_x^2\right)-\hat{z}^*,$$
$$\hat{v}_1^*=(\hat{z}^*+\hat{v}_2^*+K)(u_0)_x.$$

Under hypotheses and definitions, we have
\begin{equation}\label{el12}\delta^2 J(u_0,\varphi,\varphi) \geq 0,\; \forall \varphi \in C^1_c((0,L))\end{equation}
and there exist $r,r_1,r_2>0$ such that
\begin{eqnarray}
J(u_0)&=& \inf_{u \in B_r(u_0)} J(u) \nonumber \\ &=& \sup_{v^* \in B_{r_2}(\hat{v}^*)\cap A^*}\left\{\inf_{z^* \in B_{r_1}(\hat{z}^*)}\left\{J^*(v^*,z^*)
\right\}\right\} \nonumber \\ &=& J^*(\hat{v}^*,\hat{z}^*).
\end{eqnarray}
\end{thm}
\begin{proof}
Observe that
\begin{eqnarray}\label{el1}\frac{\partial^2 J^*(\hat{v}^*,\hat{z}^*)}{\partial (z^*)^2}&=&  \frac{2}{EA}-\frac{(\hat{v}_1^*)^2}{(\hat{v}_2^*+\hat{z}^*+\frac{EA}{2})^3}
-\frac{1}{EA} \nonumber \\ &=& \frac{1}{EA}-\frac{(u_0)_x^2}{\hat{v}_2^*+\hat{z}^*+\frac{EA}{2}}.
\end{eqnarray}

Also,
\begin{eqnarray} \hat{v}_2^*+\hat{z}^*+\frac{EA}{2}&=& EA\left( (u_0)_x +\frac{1}{2} (u_0)_x^2\right)+\frac{EA}{2}
\nonumber \\ &>& \frac{EA}{2}-EA\left(\frac{1}{4}+\frac{1}{2}\frac{1}{16}\right)
\nonumber \\ &=& EA\left(\frac{1}{4}-\frac{1}{32}\right) \nonumber \\ &=& EA \frac{7}{32}.
\end{eqnarray}

From this and (\ref{el1}), we obtain
\begin{eqnarray}\label{el2}\frac{\partial^2 J^*(\hat{v}^*,\hat{z}^*)}{\partial (z^*)^2}&>& \frac{1}{EA}-\frac{1}{16} \frac{32}{(7\;EA)}
\nonumber \\ &=& \frac{1}{EA}\left(1-\frac{2}{7}\right) \nonumber \\ &=& \frac{5}{7 EA} \nonumber \\ &>& 0, \text{ in } \Omega.
\end{eqnarray}

Thus, we may infer that there exists $r_1,r_2>0$ such that $J^*(v^*,z^*)$ is convex in $z^*$ and concave in $v^*$, on
$$B_{r_1}(\hat{z}^*) \times B_{r_2}(\hat{v}^*).$$

Now, denoting $$\hat{J}(v^*,z^*,u)=J^*(v^*,z^*)-\langle u,(v_1^*)_x+(v_2^*)_x+P \rangle_{L^2}$$ we obtain
\begin{eqnarray}
\frac{\partial \hat{J}^*(\hat{v}^*, \hat{z}^*, u_0)}{\partial z^*}&=& \frac{z^*}{K}+\frac{1}{2} \frac{(\hat{v}_1^*)^2}{(\hat{v}_2^*+\hat{z}^*+K)^2}-\frac{\hat{v}_2^*+\hat{z}^*}{EA} \nonumber \\ &=&
(u_0)_x +\frac{1}{2} (u_0)_x^2-\frac{EA\left((u_0)_x +\frac{1}{2} (u_0)_x^2\right)}{EA} \nonumber \\ &=& 0, \text{ in } \Omega.
\end{eqnarray}

Also,
\begin{eqnarray}
\frac{\partial \hat{J}^*(\hat{v}^*, \hat{z}^*, u_0)}{\partial v_1^*}&=& -\frac{(\hat{v}_1^*)}{(\hat{v}_2^*+\hat{z}^*+K)}+(u_0)_x
\nonumber \\ &=& 0, \text{ in } \Omega,
\end{eqnarray}
and
\begin{eqnarray}
\frac{\partial \hat{J}^*(\hat{v}^*, \hat{z}^*, u_0)}{\partial v_2^*}&=& \frac{1}{2} \frac{(v_1^*)^2}{(\hat{v}_2^*+\hat{z}^*+K)^2}-\frac{\hat{v}_2^*+\hat{z}^*}{EA} +(u_0)_x\nonumber \\ &=&
(u_0)_x +\frac{1}{2} (u_0)_x^2-\frac{EA\left((u_0)_x +\frac{1}{2} (u_0)_x^2\right)}{EA} \nonumber \\ &=& 0, \text{ in } \Omega.
\end{eqnarray}

Finally,
\begin{eqnarray}
\frac{\partial \hat{J}^*(\hat{v}^*, \hat{z}^*, u_0)}{\partial u}&=& -(\hat{v}_1^*)_x-(\hat{v}_2^*)_x-P \nonumber \\ &=& \delta J(u_0)
\nonumber \\ &=& \mathbf{0}, \text{ in } \Omega.
\end{eqnarray}

These last four results may be summarized by the equation
$$\delta \hat{J}^*(\hat{v}^*,\hat{z}^*,u_0)=\mathbf{0}.$$
Since above we have obtained that $J^*(v^*,z^*)$ is convex in $z^*$ and concave in $v^*$, on
$$B_{r_1}(\hat{z}^*) \times B_{r_2}(\hat{v}^*),$$  from this, the last result and from the min-max theorem, we have

\begin{eqnarray}\label{el10a} J^*(\hat{v}^*,\hat{z}^*)&=& \hat{J}^*(\hat{v}^*,\hat{z}^*,u_0) \nonumber \\ &=& \sup_{v^* \in B_{r_2}(\hat{v}^*)}\left\{\inf_{z^* \in B_{r_1}(\hat{z}^*)}\left\{\hat{J}^*(v^*,z^*,u_0)\right\}\right\} \nonumber \\ &=& \sup_{v^* \in B_{r_2}(\hat{v}^*) \cap A^*}\left\{\inf_{z^* \in B_{r_1}(\hat{z}^*)} \hat{J}^*(v^*,z^*)\right\}.
\end{eqnarray}

At this point we observe that
\begin{eqnarray}\label{el8}
J^*(\hat{v}^*,\hat{z}^*)&=& \hat{J}^*(\hat{v}^*,\hat{z}^*,u_0) \nonumber \\ &=& F^*(\hat{z}^*)-G^*_K(\hat{v}^*,\hat{z}^*) \nonumber \\ &&+\langle (u_0)_x, \hat{v}_1^*+\hat{v}_2^* \rangle_{L^2}-\langle u_0,P\rangle_{L^2}
\nonumber \\ &=&\frac{K}{2} \int_0^L (u_0)_x^2\;dx-\langle (u_0)_x,\hat{z}^* \rangle_{L^2}
\nonumber \\ &&-\langle (u_0)_x, \hat{v}_1^*+\hat{v}_2^*+\hat{z}^* \rangle_{L^2} \nonumber \\ &&+G((u_0)_x)+\frac{K}{2} \int_0^L(u_0)_x^2\;dx
\nonumber \\ &&+\langle (u_0)_x, \hat{v}_1^*+\hat{v}_2^* \rangle_{L^2}-\langle u_0,P\rangle_{L^2} \nonumber \\ &=& G((u_0)_x)-\langle u_0,P\rangle_{L^2}
\nonumber \\ &=& J(u_0).
\end{eqnarray}
On the other hand, since $\hat{v}^* \in A^*$, we may write
\begin{eqnarray}\label{el5} J^*(\hat{v}^*,\hat{z}^*)&=& \hat{J}^*(\hat{v}^*,\hat{z}^*) \nonumber \\  &=&\inf_{z^* \in B_{r_1}(\hat{z}^*)} \hat{J}^*(\hat{v}^*,z^*)
\nonumber \\ &\leq& F^*(z^*)-G^*_K(\hat{v}^*,z^*) \nonumber \\ &&+\langle (u)_x, \hat{v}_1^*+\hat{v}_2^* \rangle_{L^2}-\langle u,P\rangle_{L^2}
\nonumber \\ &\leq& F^*(z^*)-\langle (u)_x, \hat{v}_1^*+\hat{v}_2^*+z^* \rangle_{L^2} \nonumber \\ &&+G(u_x)+K \int_0^L \frac{u_x^2}{2}\;dx
\nonumber \\ &&+\langle (u)_x, \hat{v}_1^*+\hat{v}_2^* \rangle_{L^2}-\langle u_,P\rangle_{L^2},
\end{eqnarray}
$\forall u \in U,\;z^* \in B_{r_1}(\hat{z}^*)$.

In particular, there exists $r>0$ such that if $u \in B_r(u_0)$ then $z^*=Ku_x \in B_{r_1}(\hat{z}^*)$, so that from this and (\ref{el5}), we obtain
\begin{eqnarray}\label{el7} J^*(\hat{v}^*,\hat{z}^*)&\leq& - K \int_0^L \frac{u_x^2}{2}\;dx+G(u_x)+K \int_0^L \frac{u_x^2}{2}\;dx-\langle u,P\rangle_{L^2}
\nonumber \\ &=& G(u_x)-\langle u,P\rangle_{L^2} \nonumber \\ &=& J(u),\; \forall u \in B_r(u_0).\end{eqnarray}

Finally, from (\ref{el10a}), (\ref{el8}), and  (\ref{el7}), we may infer that
\begin{eqnarray}
J(u_0)&=& \inf_{u \in B_r(u_0)} J(u) \nonumber \\ &=& \sup_{v^* \in B_{r_2}(\hat{v}^*)\cap A^*}\left\{\inf_{z^* \in B_{r_1}(\hat{z}^*)}\left\{J^*(v^*,z^*)
\right\}\right\} \nonumber \\ &=& J^*(\hat{v}^*,\hat{z}^*).
\end{eqnarray}

From the first equation in such a result we may also obtain the standard second order necessary condition indicated in (\ref{el12}).

The proof is complete.
\end{proof}

\section{ The primal variational formulation for the three-dimensional model}
At this point we start to describe the primal formulation for the three-dimensional model.

Consider $\Omega \subset \mathbb{R}^3$ an open, bounded, connected set,
which represents the reference volume of an elastic solid
under the loads $f \in C(\Omega;\mathbb{R}^3)$ and the boundary loads $\hat{f} \in C(\Gamma;\mathbb{R}^3)$, where $\Gamma$  denotes
the boundary of $\Omega$. The field of displacements resulting from the actions
of $f$  and $\hat{f}$ is denoted by $u \equiv (u_1,u_2,u_3) \in U$, where
$u_1,u_2,$ and $u_3$ denotes the displacements relating the
directions $x, y,$ and $z$ respectively, in the cartesian system
$(x,y,z)$.

Here $U$ is defined by
\begin{equation}
U=\{u=(u_1,u_2,u_3) \in C^1(\overline{\Omega};\mathbb{R}^3) \; | \;
u=(0,0,0)\equiv \mathbf{0}  \text{ on } \Gamma_0\}
\end{equation}
and $\Gamma=\Gamma_0 \cup \Gamma_1$, $\Gamma_0 \cap \Gamma_1=
\emptyset$. We assume $|\Gamma_0|>0$ where $|\Gamma_0|$ denotes the Lebesgue measure of $\Gamma_0.$

The stress tensor is denoted by $\{\sigma_{ij}\}$, where
\begin{gather}
\sigma_{ij}=H_{ijkl}\left(\frac{1}{2}(u_{k,l}+u_{l,k}
+u_{m,k}u_{m,l})\right),
\end{gather}
 $$\{H_{ijkl}\}=\{\lambda \delta_{ij} \delta_{kl}+\mu(\delta_{ik}\delta_{jl}+\delta_{il}\delta_{jk})\},$$ $\{\delta_{ij}\}$ is the Kronecker delta and $\lambda,\mu>0$ are the Lam\'{e} constants (we assume they are such that $\{H_{ijkl}\}$ is a symmetric constant  positive definite forth order tensor).
 Here, $i,j,k,l \in \{1,2,3\}.$

 The boundary value form of the non-linear
elasticity model is given by
\begin{gather}\label{9.9.10.1}
\left \{
\begin{array}{ll}
 \sigma_{ij,j}+(\sigma_{mj}u_{i,m})_{,j}+f_i=0, &  \text{ in } \Omega,
 \\
  u= \mathbf{0}, & \text{ on }\Gamma_0,
\\
\sigma_{ij}n_j+\sigma_{mj}u_{i,m}n_j=\hat{f}_i, & \text{ on } \Gamma_1,
  \end{array} \right.\end{gather}
 where $\textbf{n}=(n_1,n_2,n_3)$ denotes the outward normal to the surface $\Gamma.$

The corresponding primal variational formulation is represented by
$J:U \rightarrow \mathbb{R}$, where
\begin{eqnarray}
J(u)&=&\frac{1}{2}\int_\Omega H_{ijkl}\left(\frac{1}{2}(u_{i,j}+u_{j,i}
+u_{m,i}u_{m,j})\right)\left(\frac{1}{2}(u_{k,l}+u_{l,k}+
u_{m,k}u_{m,l})\right)\;dx\nonumber \\ &&-\langle u,f \rangle_{L^2(\Omega;\mathbb{R}^3)}-\int_{\Gamma_1} \hat{f}_i u_i \;d\Gamma
\end{eqnarray}
where $$\langle u,f \rangle_{L^2(\Omega;\mathbb{R}^3)}=\int_\Omega f_i u_i \;dx.$$
\begin{obe} By a regular Lipschitzian boundary $\Gamma$ of $\Omega$ we mean regularity enough so that the standard Gauss-Green formulas of integrations by parts to hold.
Also, we denote by $\mathbf{0}$ the zero vector in  appropriate function spaces.

About the references, similarly as for the one-dimensional case, we refer to \cite{10,11,2900,85} as the first articles to deal with the convex analysis approach applied to non-convex and non-linear mechanics models. Indeed, the present work complements such important original publications, since in these previous results the complementary energy is established as a perfect duality principle for the case of positive definiteness of the stress tensor (or the membrane force tensor, for plates and shells models) at a critical point.

We have relaxed such constraints, allowing to some extent, the stress tensor to not  be necessarily either positive or negative definite in $\Omega$.
Similar problems and models are addressed in \cite{120}.

Moreover, we highlight again that existence results for models in elasticity are addressed in \cite{903,3,4}. Finally, the standard tools of convex analysis here used may be found in
\cite{6,12,29,120}.
\end{obe}

\section{The main duality principle for the three-dimensional model}

In this section we present the main duality principle for the 3-Dimensional model.

The main result is summarized by the following theorem.
\begin{thm}Let $J:U \rightarrow \mathbb{R}$ be defined by
\begin{eqnarray}J(u)&=&\int_\Omega H_{ijkl}\left(\frac{1}{2}(u_{i,j}+u_{j,i}+ u_{m,i}u_{m,j})\right)\left(\frac{1}{2}(u_{k,l}+u_{k,l}+ u_{q,k}u_{q,l})\right) \;dx \nonumber \\ &&-\langle u_if_i \rangle_{L^2(\Omega)}-\langle u_i, \hat{f}_i \rangle_{L^2(\Gamma_1)}.\end{eqnarray}

Assume $u_0 \in \hat{U}$ is such that $$\delta J(u_0)=\mathbf{0},$$
where $$\hat{U}=\{u \in U\;:\; \|u_{i,j}\|_\infty< 1/8,\; \forall i,j \in \{1,2,3\}\}.$$

Define $F:U \rightarrow \mathbb{R}$ by $$F(u_{i,j})=\frac{K}{2}\int_\Omega \left(\frac{u_{i,j}+u_{j,i}}{2}\right)\left(\frac{u_{i,j}+u_{j,i}}{2}\right)  \;dx,$$

$G_K:U \rightarrow \mathbb{R}$ by
$$G_K(\{u_{i,j}\})=G(u)+\frac{K}{2}\int_\Omega \left(\frac{u_{i,j}+u_{j,i}}{2}\right)\left(\frac{u_{i,j}+u_{j,i}}{2}\right)  \;dx $$
where $$G(u)=\int_\Omega H_{ijkl}\left(\frac{1}{2}(u_{i,j}+u_{j,i}+ u_{m,i}u_{m,j})\right)\left(\frac{1}{2}(u_{k,l}+u_{k,l}+ u_{q,k}u_{q,l})\right) \;dx,$$
and
$J^*:V \times V \times V \rightarrow \mathbb{R}$ by
$$J^*(v^*,z^*)=F^*(z^*)-G^*_K(v^*,z^*),$$
where
$$V=C(\overline{\Omega};\mathbb{R}^{3 \times 3}),$$
\begin{eqnarray}F^*(z^*)&=& \sup_{v \in V} \{\langle v_{ij},z^*_{ij}\rangle_{L^2}-F(v)\} \nonumber \\ &=&
\frac{1}{2K} \int_\Omega z^*_{ij}z^*_{ij}\;dx\end{eqnarray}
and
\begin{eqnarray} G^*_K(v^*,z^*)&=& \sup_{ (v_1,v_2) \in V \times V} \left\{
\langle (v_1)_{ij},z^*_{ij}+(v_2^*)_{ij}\rangle_{L^2}+\langle (v_2)_{ij},(v_1^*)_{ij}\rangle_{L^2} \right.\nonumber \\
&&-\frac{1}{2}\int_\Omega H_{ijkl}\left((v_1)_{ij}+\frac{1}{2}(v_2)_{mi}(v_2)_{mj}\right)\left((v_1)_{kl}+\frac{1}{2}(v_2)_{qk}(v_2)_{ql}\right)\;dx
\nonumber \\ &&\left.- \frac{K}{2}\int_\Omega (v_2)_{ij} (v_2)_{ij}\;dx\right\}
\nonumber \\ &=& \frac{1}{2}\int_\Omega \overline{((v_2^*)_{ij}+(z^*)_{ij}+K \delta_{ij})} (v_1^*)_{mi}(v_1^*)_{mj};dx
\nonumber \\ &&+\frac{1}{2}\int_\Omega \overline{H}_{ijkl} ((v_2^*)_{ij}+z^*_{ij})((v_2^*)_{kl}+z^*_{kl})\;dx,\end{eqnarray}
if $\{(v_2^*)_{ij}+(z^*)_{ij}+K \delta_{ij}\} \text{ is positive definite in } \Omega.$

Here $$\{\overline{(v_2^*)_{ij}+(z^*)_{ij}+K \delta_{ij}}\}=\{(v_2^*)_{ij}+(z^*)_{ij}+K \delta_{ij}\}^{-1}.$$

Define also, $A^*=A_1 \cap A_2$, where
$$A_1=\{v^*=(v_1^*,v_2^*) \in V \times V\;:\; (v_1^*)_{ij,j}+(v_2^*)_{ij,j}+f_i=0,\text{ in } \Omega\},$$
and
$$A_2=\{v^*=(v_1^*,v_2^*) \in V \times V\;:\; (v_1^*)_{ij}n_j+(v_2^*)_{ij}n_j-\hat{f}_i=0,\text{ on } \Gamma_t\},$$
and let $K>0$ be such that
$$M=\left\{ \frac{ D_{ijkl}}{K}-\frac{3}{32K}\delta_{ij}-\overline{H}_{ijkl}\right\}$$ is a positive definite tensor,
where \begin{equation}D_{ijkl}=\left\{\begin{array}{ll}
1,& \text{ if } i=k \text{ and } j=l, \\
0,& \text{ otherwise }\end{array} \right. \end{equation}
and, in an appropriate sense,
$$\{\overline{H}_{ijkl}\}=\{H_{ijkl}\}^{-1}.$$

Assume also,
$$\hat{z}^*_{ij}=K \left(\frac{(u_0)_{i,j}+(u_0)_{j,i}}{2}\right),$$
$$(\hat{v}_2^*)_{ij}=H_{ijkl}\left(\frac{(u_0)_{k,l}+(u_0)_{l,k}}{2}+\frac{1}{2}(u_0)_{m,k}(u_0)_{m,l}\right)-(\hat{z}^*)_{ij},$$
$$(\hat{v}_1^*)_{ij}=((\hat{z}^*)_{im}+(\hat{v}_2^*)_{im}+K\delta_{im})(u_0)_{m,j},$$
$\forall i,j \in \{1,2,3\}$ and $K>0$ is also such that
$$\{(\hat{v}_2^*)_{ij}+(\hat{z}^*)_{ij}+K \delta_{ij}\} \geq K\{\delta_{ij}\}/2$$

Under hypotheses and definitions,
 there exist $r,r_1,r_2>0$ such that
\begin{eqnarray}
J(u_0)&=& \inf_{u \in B_r(u_0)} J(u) \nonumber \\ &=& \sup_{v^* \in B_{r_2}(\hat{v}^*)\cap A^*}\left\{\inf_{z^* \in B_{r_1}(\hat{z}^*)}\left\{J^*(v^*,z^*)
\right\}\right\} \nonumber \\ &=& J^*(\hat{v}^*,\hat{z}^*).
\end{eqnarray}
\end{thm}
\begin{proof}
Observe that, denoting $$\{\overline{\overline{(\hat{v}_2^*)_{ij}+\hat{z}^*_{ij}+K \delta_{ij}}}\}=\{(\hat{v}_2^*)_{ij}+\hat{z}^*_{ij}+K \delta_{ij}\}^{-2}$$ and
$$\{\overline{\overline{\overline{(\hat{v}_2^*)_{ij}+\hat{z}^*_{ij}+K \delta_{ij}}}}\}=\{(\hat{v}_2^*)_{ij}+\hat{z}^*_{ij}+K \delta_{ij}\}^{-3}$$
we have
\begin{eqnarray}\label{el10}\left\{\frac{\partial^2 J^*(\hat{v}^*,\hat{z}^*)}{\partial z^*_{ij} \partial z^*_{kl}}\right\}&=&  \left\{\frac{D_{ijkl}}{K}-(\overline{\overline{\overline{(\hat{v}_2^*)_{ij}+(\hat{z}^*)_{ij}+K \delta_{ij}}}})(\hat{v}_1^*)_{mk}(\hat{v}_1^*)_{ml}
-\overline{H}_{ijkl}\right\} \nonumber \\ &=& \left\{\frac{D_{ijkl}}{K}-(\overline{(\hat{v}_2^*)_{ij}+(\hat{z}^*)_{ij}+K \delta_{ij}})(u_0)_{mk}(u_0)_{ml}
-\overline{H}_{ijkl}\right\} \nonumber \\ &\geq& \left\{\frac{D_{ijkl}}{K}-\frac{3}{32 K} \delta_{ij}
-\overline{H}_{ijkl}\right\} \nonumber \\ &>& \mathbf{0}.
\end{eqnarray}

Thus, there exist $r_1,r_2>0$ such that $J^*(v^*,z^*)$ is convex in $z^*$ and concave in $v^*$ on
$$B_{r_1}(\hat{z}^*) \times B_{r_2}(\hat{v}^*).$$

Now, denoting \begin{eqnarray}\hat{J}(v^*,z^*,u)&=&J^*(v^*,z^*)-\langle u,(v_1^*)_{ij,j}+(v_2^*)_{ij,j}
+f_i \rangle_{L^2(\Omega)}\nonumber \\ &&+\langle u,(v_1^*)_{ij}n_j+(v_2^*)_{ij}n_j-\hat{f}_i \rangle_{L^2(\Gamma_t)}\end{eqnarray}
 we obtain
\begin{eqnarray}
\frac{\partial \hat{J}^*(\hat{v}^*, \hat{z}^*, u_0)}{\partial (z^*)_{ij}}&=& \frac{z^*_{ij}}{K}+\frac{1}{2} \overline{\overline{(\hat{v}_2^*)_{ij}+(\hat{z}^*)_{ij}+K \delta_{ij}}})(\hat{v}_1^*)_{mi}(\hat{v}_1^*)_{mj} -\overline{H}_{ijkl}((\hat{v}_2^*)_{kl}+z^*_{kl})
\nonumber \\ &=&\frac{(u_0)_{i,j}+(u_0)_{j,i}}{2} +\frac{1}{2}(u_0)_{mi}(u_0)_{mj} -\overline{H}_{ijkl}((\hat{v}_2^*)_{kl}+z^*_{kl}) \nonumber \\
&=& 0, \text{ in } \Omega.
\end{eqnarray}

Also,
\begin{eqnarray}
\frac{\partial \hat{J}^*(\hat{v}^*, \hat{z}^*, u_0)}{\partial (v_1^*)_{ij}}&=& -\overline{(\hat{v}_2^*)_{im}+(\hat{z}^*)_{im}+K \delta_{im})}(\hat{v}_1^*)_{mj}
+(u_0)_{i,j}\nonumber \\ &=& 0, \text{ in } \Omega,
\end{eqnarray}
and
\begin{eqnarray}
\frac{\partial \hat{J}^*(\hat{v}^*, \hat{z}^*, u_0)}{\partial (v_2^*)_{ij}}&=& \frac{1}{2} \overline{\overline{(\hat{v}_2^*)_{ij}+(\hat{z}^*)_{ij}+K \delta_{ij}}})(\hat{v}_1^*)_{mi}(\hat{v}_1^*)_{mj} -\overline{H}_{ijkl}((\hat{v}_2^*)_{kl}+\hat{z}^*_{kl})
\nonumber \\ &&+\frac{(u_0)_{i,j}+(u_0)_{j,i}}{2}
\nonumber \\ &=&\frac{(u_0)_{i,j}+(u_0)_{j,i}}{2} +\frac{1}{2}(u_0)_{mi}(u_0)_{mj} -\overline{H}_{ijkl}((\hat{v}_2^*)_{kl}+\hat{z}^*_{kl}) \nonumber \\
&=& 0, \text{ in } \Omega.
\end{eqnarray}

Finally,
\begin{eqnarray}
\frac{\partial \hat{J}^*(\hat{v}^*, \hat{z}^*, u_0)}{\partial u}&=& \left \{
\begin{array}{ll}
 -(\hat{v}_1^*)_{ij,j}-(\hat{v}_2^*)_{ij,j}-f_i, &  \text{ in } \Omega,
 \\
 (\hat{v}_1^*)_{ij}n_j+(\hat{v}_2^*)_{ij}n_j-\hat{f}_i, & \text{ on } \Gamma_1,
  \end{array} \right.
\end{eqnarray}
Hence, $$\frac{\partial \hat{J}^*(\hat{v}^*, \hat{z}^*, u_0)}{\partial u}=\delta J(u_0)=\mathbf{0}.$$

These last four results may be summarized by the equation
$$\delta \hat{J}^*(\hat{v}^*,\hat{z}^*,u_0)=\mathbf{0}.$$
Since above we have obtained that $J^*(v^*,z^*)$ is convex in $z^*$ and concave in $v^*$, on
$$B_{r_1}(\hat{z}^*) \times B_{r_2}(\hat{v}^*),$$  from this, the last result and from the min-max theorem, we have

\begin{eqnarray}\label{el100} J^*(\hat{v}^*,\hat{z}^*)&=& \hat{J}^*(\hat{v}^*,\hat{z}^*,u_0) \nonumber \\ &=& \sup_{v^* \in B_{r_2}(\hat{v}^*)}\left\{\inf_{z^* \in B_{r_1}(\hat{z}^*)}\left\{\hat{J}^*(v^*,z^*)\right\}\right\} \nonumber \\ &=& \sup_{v^* \in B_{r_2}(\hat{v}^*) \cap A^*}\left\{\inf_{z^* \in B_{r_1}(\hat{z}^*)} \hat{J}^*(v^*,z^*)\right\}.
\end{eqnarray}

At this point we observe that
\begin{eqnarray}\label{el80}
J^*(\hat{v}^*,\hat{z}^*)&=& \hat{J}^*(\hat{v}^*,\hat{z}^*,u_0) \nonumber \\ &=& F^*(\hat{z}^*)-G^*_K(\hat{v}^*,\hat{z}^*) \nonumber \\ &&+\langle (u_0)_{i,j}, (\hat{v}_1^*)_{ij}+(\hat{v}_2^*)_{ij} \rangle_{L^2}-\langle (u_0)_i,f_i\rangle_{L^2(\Omega)}-\langle (u_0)_i,\hat{f}_i\rangle_{L^2(\Gamma_1)}
\nonumber \\ &=&-F(\{(u_0)_{i,j}\})+\langle (u_0)_{i,j},(\hat{z}^*)_{ij} \rangle_{L^2}
\nonumber \\ &&-\langle (u_0)_{i,j}, (\hat{v}_1^*)_{ij}+(\hat{v}_2^*)_{ij}+(\hat{z}^*)_{ij} \rangle_{L^2} \nonumber \\ &&+G((u_0))+\frac{K}{2}\int_\Omega \left(\frac{(u_0)_{i,j}+(u_0)_{j,i}}{2}\right)\left(\frac{(u_0)_{i,j}+(u_0)_{j,i}}{2}\right)  \;dx
\nonumber \\ &&+\langle (u_0)_{i,j}, (\hat{v}_1^*)_{ij}+(\hat{v}_2^*)_{ij} \rangle_{L^2}-\langle (u_0)_i,f_i\rangle_{L^2(\Omega)}-\langle (u_0)_i,\hat{f}_i\rangle_{L^2(\Gamma_1)} \nonumber \\ &=& G((u_0))-\langle (u_0)_i,f_i\rangle_{L^2(\Omega)}-\langle (u_0)_i,\hat{f}_i\rangle_{L^2(\Gamma_1)}
\nonumber \\ &=& J(u_0).
\end{eqnarray}
On the other hand, since $\hat{v}^* \in A^*$, we may write
\begin{eqnarray}\label{el50} J^*(\hat{v}^*,\hat{z}^*)&=& \inf_{z^* \in B_{r_1}(\hat{z}^*)} \hat{J}^*(\hat{v}^*,z^*)
\nonumber \\ &\leq& F^*(z^*)-G^*_K(\hat{v}^*,z^*) \nonumber \\ &&+\langle (u)_{i,j}, (\hat{v}_1^*)_{ij}+(\hat{v}_2^*)_{ij} \rangle_{L^2}
-\langle u_i,f_i\rangle_{L^2(\Omega)}-\langle u_i,\hat{f}_i\rangle_{L^2(\Gamma_1)}
\nonumber \\ &\leq& F^*(z^*)-\langle u_{ij}, (\hat{v}_1^*)_{ij}+(\hat{v}_2^*)_{ij}+(z^*)_{ij} \rangle_{L^2} \nonumber \\ &&+G(u)+\frac{K}{2}\int_\Omega \left(\frac{u_{i,j}+u_{j,i}}{2}\right)\left(\frac{u_{i,j}+u_{j,i}}{2}\right)  \;dx
\nonumber \\ &&+\langle u_{i,j}, (\hat{v}_1^*)_{ij}+(\hat{v}_2^*)_{ij} \rangle_{L^2}
-\langle u_i,f_i\rangle_{L^2(\Omega)}-\langle u_i,\hat{f}_i\rangle_{L^2(\Gamma_1)}
\nonumber \\ &=& F^*(z^*)-\langle u_{i,j},z^*_{ij} \rangle_{L^2} \nonumber \\ &&+G(u)+\frac{K}{2}\int_\Omega \left(\frac{u_{i,j}+u_{j,i}}{2}\right)\left(\frac{u_{i,j}+u_{j,i}}{2}\right)  \;dx
\nonumber \\ &&
-\langle u_i,f_i\rangle_{L^2(\Omega)}-\langle u_i,\hat{f}_i\rangle_{L^2(\Gamma_1)},
\end{eqnarray}
$\forall u \in U,\;z^* \in B_{r_1}(\hat{z}^*)$.

In particular, there exists $r>0$ such that if $u \in B_r(u_0)$ then $\{z^*_{ij}\}=K\left\{\frac{u_{i,j}+u_{j,i}}{2}\right\} \in B_{r_1}(\hat{z}^*)$, so that from this and (\ref{el50}), we obtain
\begin{eqnarray}\label{el70} J^*(\hat{v}^*,\hat{z}^*)&\leq&-\frac{K}{2}\int_\Omega \left(\frac{u_{i,j}+u_{j,i}}{2}\right)\left(\frac{u_{i,j}+u_{j,i}}{2}\right)  \;dx +G(u) \nonumber \\ && +\frac{K}{2}\int_\Omega \left(\frac{u_{i,j}+u_{j,i}}{2}\right)\left(\frac{u_{i,j}+u_{j,i}}{2}\right)  \;dx -\langle u_i,f_i\rangle_{L^2(\Omega)}-\langle u_i,\hat{f}_i\rangle_{L^2(\Gamma_1)}
\nonumber \\ &=& G(u)-\langle u_i,f_i\rangle_{L^2(\Omega)}-\langle u_i,\hat{f}_i\rangle_{L^2(\Gamma_1)} \nonumber \\ &=& J(u),\; \forall u \in B_r(u_0).\end{eqnarray}

Finally, from (\ref{el100}), (\ref{el80}), and  (\ref{el70}), we may infer that
\begin{eqnarray}
J(u_0)&=& \inf_{u \in B_r(u_0)} J(u) \nonumber \\ &=& \sup_{v^* \in B_{r_2}(\hat{v}^*)\cap A^*}\left\{\inf_{z^* \in B_{r_1}(\hat{z}^*)}\left\{J^*(v^*,z^*)
\right\}\right\} \nonumber \\ &=& J^*(\hat{v}^*,\hat{z}^*).
\end{eqnarray}

The proof is complete.
\end{proof}

\section{Conclusion} In this article we develop some theoretical  results on duality for a class of non-convex optimization problems in elasticity.
In this first approach we have developed in details  duality principles and sufficient optimality conditions for  local minimality
for  one and three-dimensional  models in elasticity. It is worth mentioning the results may be extended to other models in elasticity and to other models of plates and shells.

\end{document}